\documentclass[10pt,a4paper]{article}
\usepackage[utf8]{inputenc}
 \usepackage[T1]{fontenc}
 
\usepackage{amsfonts}
\usepackage{amssymb}
\usepackage{amsthm}
\usepackage{float}
 \usepackage{booktabs}
\usepackage[numbers]{natbib}
\usepackage{graphicx}

\usepackage{authblk}

\usepackage{amsmath}
\usepackage{blindtext}

\newtheorem{Theorem}{Theorem}
\newcommand{\twofractals}[4]{%
\begin{center}
\begin{minipage}[t]{0.45\textwidth}
\begin{centering}
\includegraphics[width=0.98\textwidth]{#1}
\end{centering}
\centering{\small  #2}
\end{minipage}\hfill%
\begin{minipage}[t]{0.45\textwidth}
\begin{centering}
\includegraphics[width=0.98\textwidth]{#3}
\end{centering}
\centering{\small #4}
\end{minipage}\hfill%
\end{center}%
}
\newcommand{\imagen}[4]{%
\begin{center}
\begin{minipage}[t]{0.45\textwidth}
\begin{centering}
\includegraphics[width=0.88\textwidth]{#1}
\end{centering}
\centering{\small  #2}
\end{minipage}\hfill%
\begin{minipage}[t]{0.45\textwidth}
\begin{centering}
\includegraphics[width=0.88\textwidth]{#3}
\end{centering}
\centering{\small #4}
\end{minipage}\hfill%
\end{center}%
}

\title{A characterization of the dynamics of Schröder’s method for polynomials with two roots}

\date{November 2020}
\author[1*]{José M. Gutierrez}
\author[2]{Víctor Galilea Martín}
\affil[1]{Department of Mathematics and Computing, University of La Rioja. 53 Madre de Dios street, 26006, Logroño, Spain. E-mail:  jmguti@unirioja.es}
\affil[2]{Department of Mathematics and Computing, University of La Rioja. 53 Madre de Dios street, 26006, Logroño, Spain. E-mail: victor.galilea@unirioja.es }
\affil[*]{Corresponding author.}

\begin{document}
\maketitle
\newpage
\section{Introduction}

In his seminal paper, published in 1870, about the solution of a nonlinear equation in a single unknown, 
\begin{equation} 
f(z)=0,
\end{equation}\label{eq1}
Schr\"oder deals with the problem of characterizing general iterative algorithms to solve~\eqref{eq1} with a prefixed order of convergence $\omega\ge 2$ (see the original paper \cite{Sch} or the commented English translation \cite{Ste}). The main core of Schr\"oder's work studies two families of iterative processes, the well-known families of first and second kind \cite{Pet1}, \cite{Pet2}. The $\omega$-th member of  these families is an iterative method that converges with order $\omega$ to a solution of~\eqref{eq1}. In this way, the second method of both families is Newton's method. The third method of the first family is Chebyshev's method. The third method of the second family is Halley's method. The rest of the methods in both families (with order of convergence $\omega\ge 4$) are not so well known. 

Note that Newton's, Chebyshev's and Halley's method are also members of another famous family of iterative methods, the known as Chebyshev-Halley family of methods (introduced by Werner \cite{Wer} and reported by may other authors such as \cite{Amat} or \cite{Dub}):
 \begin{equation}\label{eq2}
z_{k+1}=z_k-\left(1-\frac{1}{2}\frac{L_f(z_k)}{1-\alpha L_f(z_k)}\frac{f(z_k)}{f'(zk+1)}\right), \quad \alpha\in \mathbb{R},\quad n\ge 0, \quad z_0\in\mathbb{C},
\end{equation}
where we have used the notation
 \begin{equation}\label{eq3}
 L_f(z)=\frac{f(z)f''(z)}{f'(z)^2}.
\end{equation}
In fact, Chebyshev's method is obtained for $\alpha=0$, Halley's method appears for $\alpha=1/2$ and Newton's method can be obtained as a limit case when $|\alpha| \to \infty$. Except for the limit case of Newton's method, all the methods in the family have third order of convergence.

In this general context of families of iterative methods for solving nonlinear equations, we would like to highlight a detail that appears in the aforementioned paper by Schr\"oder~\cite{Sch}.  Actually, in the third section of this article, Schr\"oder constructs an algorithm by applying Newton's method to the equation 
$$
\frac{f(z)}{f'(z)}=0.
$$
The resulting iterative scheme can be written as
$$
z_{k+1}=z_k-\frac{f(z_k)f'(z_k)}{f'(z_k)^2-f(z_k)f''(z_k)}, \quad n\ge 0, \quad z_0\in\mathbb{C},
$$
as it is known as \emph{Schr\"oder's method} by many authors (see for instance \cite{Proinov} or  \cite{Scavo}).

For our convenience, we denote by $S_f(z)$ the iteration function of Schr\"oder's method. Note that it can be written in terms of the function $L_f(z)$ introduced in~\eqref{eq3} in the following way
 \begin{equation}\label{eq4}
z_{k+1}=S_f(z_k)=z_k-\frac{1}{1-L_f(z_k)}\frac{f(z_k)}{f'(z_k)} \quad n\ge 0, \quad z_0\in\mathbb{C}.
\end{equation}
The same Schr\"oder compares the resulting algorithm~\eqref{eq4} with Newton's method and says: 
\begin{quote}
``It is an equally worthy algorithm which to my knowledge has not been previously considered. Besides, being almost as simple, this latter algorithm has the advantage that it converges quadratically even for multiple roots''.
\end{quote}

Curiously, Schr\"oder's method~\eqref{eq4} does not belong neither to the Schr\"oder's families of first and second kind nor the Chebyshev-Halley family~\eqref{eq2}. It has very interesting numerical properties, such as the quadratic convergence even for multiple roots, but the fact of having a high computational cost (equivalent to the the third order methods in~\eqref{eq2}) could be an important handicap for practical purposes.

In this paper we present a first  approach to the dynamical behavior of Schr\"oder's method. So we show that for polynomials with two different roots and different multiplicities, it is possible to characterize the basins of attraction and the corresponding Julia set. We can appreciate the influence of the multiplicities in such sets.

\section{Preliminaries}

In the 13th section of Schr\"oder's work~\cite{Sch}, that has the title \emph{The Principal Algorithms Appllied to very Simple Examples}, we can find the first dynamical study of a couple of rootfinding methods. Actually, Schr\"oder's considers those who, in his opinion, are the two most useful methods: Newton's method, defined by the iterative scheme
\begin{equation}\label{eq5}
z_{k+1}=N_f(z_k)=z_k-\frac{f(z_k)}{f'(z_k)} \quad n\ge 0, \quad z_0\in\mathbb{C},
\end{equation}
and the method $z_{k+1}=S_f(z_k)$ given in~\eqref{eq4}.

In the simplest case, namely equations with only one root,  we can assume without loss of generality that $f(z)=z^n$. It is easy to see that
$$
N_f(z)=\frac{n-1}{n}z,\quad S_f(z)=0.
$$
So Schr\"oder's method gives the correct root ($z=0$) of the equation in just one step, whereas Newton's method converges to this root with lineal convergence:
$$
z_k=\left(\frac{n-1}{n}\right)^k z_0.
$$
Consequently, for equations with s single root Schr\"oder concludes that the convergence regions of these two methods is the entire complex plane.

The next simple case considered by Schr\"oder is the quadratic equation. Again, without loss of generality he asumes 
$f(z)=(z-1)(z+1)$. After a series of cumbersome calculus, he estates that in this case and for both methods, the entire complex plane decomposes into two regions separated by the imaginary axis. A few years later, Cayley~\cite{Cay} addresses the same problem, only for Newton's method. In a very elegant way, Cayley proves that for polynomials
\begin{equation}\label{eq6}
f(z)=(z-a)(z-b),\quad a,b\in \mathbb{C}, \quad a\ne b,
\end{equation}
Newton's iterates converge to the root $a$ if $|z_0-a|<|z_0-b|$ and to the  root $b$ if $|z_0-b|<|z_0-a|$. The Julia set is the equidistant line between the two roots. The key to prove this result is to check that Newton iteration function~\eqref{eq5} applied to polynomials~\eqref{eq6} is conjugate via the M\"obius map
\begin{equation}\label{eq7}
M(z)=\frac{z-a}{z-b}
\end{equation}
with the function $R(z)=z^2$, that is, $R(z)=M\circ N_f\circ M^{-1}(z)$. The unit circle $S^1=\{z\in\mathbb{C}; |z|=1\}$ is invariant by $R$. Its anti-image by $R$ is the bisector between the roots $a$ and $b$.

Two functions $f,g: \mathbb{C} \to  \mathbb{C} $ are said topologically conjugate if  there exists a homeomorphism $\varphi$ such that 
$$
\varphi\circ g=f\circ \varphi.
$$
Topological conjugation is a very useful tool in dynamical systems (see \cite{Dev} for more details)
because two conjugate functions share the same dynamics properties, from the topological viewpoint. For instance, the fixed points of one function are mapped into the fixed points of the other, the periodic points of one function are mapped into the periodic points of the other function, and so on. Speaking informally, we can say that the two functions are the same from a dynamical point of view. As we have just seen, in some cases one of the functions in a conjugation could be much simpler than the other. In the case of Cayley's problem $R(z)=z^2$ is topologically conjugate (and much simpler) to
$$
N_f(z)=\frac{a b-z^2}{a+b-2 z}.
$$

In the same way, we have that Schr\"oder's method~\eqref{eq4} applied to polynomials~\eqref{eq6} 
$$
S_f(z)=\frac{z^2 (a+b)-4 a b z+a b (a+b)}{a^2-2 z (a+b)+b^2+2 z^2}
$$is conjugated with $-R(z)$ via the M\"obius map defined in~\eqref{eq7},  that is $M\circ S_f\circ M^{-1}(z)=-z^2$. Consequently, the dynamical behavior of Schr\"oder's method for quadratic polynomials mimics the behavior of Newton's method: the Julia set is the bisector between the two roots and the basins of attraction are the corresponding half-planes.

\section{Main results}

Now we consider the case of polynomials with two roots, but with different multiplicities, $m\ge n\ge 1$:
\begin{equation}\label{eq8}
f(z)=(z-a)^m(z-b)^n,\quad a,b\in \mathbb{C}, \quad a\ne b.
\end{equation}
For our simplicity, and to better appreciate symmetries, we move the roots $a$ and $b$ to $1$ and $-1$. To do this, we conjugate with the affine map
\begin{equation}\label{eq9}
A(z)=1+2\frac{z-a}{a-b}
\end{equation}
to obtain a simpler function that does not depend on the roots $a$ and $b$. Let $T_{m,n}(z)$ be the corresponding conjugate map:
\begin{equation}\label{eq10}
T_{m,n}(z)=A\circ S_f\circ A^{-1}(z)=\frac{(m-n) z^2 +2 (m+n) z +m-n}{(m+n) z^2 +2(m-n)z +m+n}.
\end{equation}
A new conjugation of $T_{m,n}(z)$ with the M\"obius map \eqref{eq7}, when $a=1$ and $b=-1$ provides us a new rational function whose dynamics are extremely simple. Actually:
\begin{equation}\label{eq11}
R_{m,n}(z)=M\circ T_{m,n}\circ M^{-1}(z)=-\frac{nz^2}{m}.
\end{equation}

Note that the circle $C_{m,n}=\{z\in\mathbb{C}; |z|=m/n\}$ is invariant by $R_{m,n}(z)$. After iteration by $R_{m,n}(z)$, the orbits of the points $z_0$ with  $|z_0|<m/n$ go to 0, whereas the orbits  of the points $z_0$ with  $|z_0|>m/n$ go to $\infty$. Consequently, $C_{m,n}$ is the Julia set of the map $R_{m,n}(z)$.

\begin{Theorem}\label{teo1}
Let $T_{m,n}(z)$ be the rational map defined by~\eqref{eq10} and let us denote by  $J_{m,n}$ its Julia set. Then we have:
\begin{enumerate}
\item   If $m=n$,  then $J_{m,m}$ is the imaginary axis.
\item   If $m>n\ge 1$, then $J_{m,n}$  is the circle
$$J_{m,n}=\left\{z\in\mathbb{C}; \left|z+\frac{m^2+n^2}{m^2-n^2}\right|=\frac{2mn}{m^2-n^2}\right\}.$$
\end{enumerate}
\end{Theorem}

\begin{proof}
The proof follows immediately, just by taking into account that $J_{m,n}$ is the pre-image by $M(z)=(z-1)/(z+1)$ of the circle $C_{m,n}$ and by distinguishing the two situations indicated in the statement of the theorem.
\end{proof}

\begin{Theorem}\label{teo2}
Let $S_f(z)$ be the rational map defined by applying Schr\"oder's method to polynomials~\eqref{eq8} and let us denote by  $J_{m,n,a,b}$ its Julia set. Then we have:
\begin{enumerate}
\item   If $m=n$,  $J_{m,m,a,b}$ is the equidistant line between the points $a$ and $b$.
\item   If $m>n\ge 1$, then $J_{m,n,a,b}$  is the circle
$$J_{m,n,a,b}=\left\{z\in\mathbb{C}; \left|z+\frac{b m^2-a n^2}{m^2-n^2}\right|=\frac{mn |a-b|}{m^2-n^2}\right\}.$$
\end{enumerate}
\end{Theorem}

\begin{proof}
Now we deduce this result  by calculating the pre-images of $J_{m,n}$ by the affine map $A(z)$ defined in~\eqref{eq9}  in the two situations indicated in the previous theorem.
\end{proof}

\section{Conclusions and further work}

We have studied the behavior of Schr\"oder's method for polynomials with two different complex roots and with different multiplicities~\ref{eq8}. Actually, we have proved that the Julia set of the corresponding rational functions obtained in this case is a circle given in Theorem~\ref{teo2}.

In addition, Theorem~\ref{teo1} gives us a universal result that characterizes the behavior of Schr\"oder's method in a very simplified form, depending only of the values of the multiplicities $m$ and $n$. The influence of the roots $a$ and $b$ is revealed in Theorem~\ref{teo2}, and is just an affine transformation of the situation given in Theorem~\ref{teo1}.

Let us consider the points $(x,y)\in\mathbb{R}^2$ given by the centers and radius of the circles defined in Theorem~\ref{teo1}, that is
$$
(x,y)=\left( \frac{m^2+n^2}{m^2-n^2}, \frac{2mn}{m^2-n^2}\right).
$$
These points belong to the hyperbola $x^2-y^2=1$ in the real plane $\mathbb{R}^2$.

In addition, we appreciate that are polynomials for which Schr\"oder's method  has the same dynamical behavior. Actually, if we introduce the new parameter
$$
p=\frac{m}{n},
$$
we have that the circles $J_{m,n}$ defined in Theorem~\ref{teo1} can be expressed as
$$J_{p}=\left\{z\in\mathbb{C}; \left|z+\frac{p^2+1}{p^2-1}\right|=\frac{2p}{p^2-1}\right\}.$$
Therefore Schr\"oder's method applied to polynomials with couples of multiplicities $(m,n)$ having the same quotient $p$ have the same Julia set.

We can schematize of the dynamics of Schr\"oder's method applied to polynomials $(z-1)^m(z+1)^n$, $m>n$ in the following way:
\begin{itemize}
\item When $p=m/n\to \infty$, the Julia sets $J_p$ are circles that tends to collapse in the point $z=-1$.
\item When $p=m/n\to 1^+$, the Julia sets $J_p$ are circles with centers in the negative real line. Note
that centers
$$
-\frac{p^2+1}{p^2-1}\to -\infty \text{ when } p\to 1^+
$$
and radius
$$
\frac{2p}{p^2-1}\to \infty \text{ when } p\to 1^+.
$$
So when $p\to 1^+$ the Julia set are circles getting bigger and tending to ``explode'' into the limit case, given by the imaginary axis when $p=1$.
\end{itemize}

If we consider the presence of the roots $a$ and $b$, the dynamics of Schr\"oder's method applied to polynomials $(z-a)^m(z-b)^n$, $m>n$ can be summarized in a ``travel'' from a circle concentrated in the root with the smallest multiplicity, $b$ to circles with the center in the line connecting the roots $a$ and $b$ and radius tending to infinity until the ``explosion'' into the limit case, given by the bisector of the two roots,  when $p=1$.

In Figures~\ref{fig1}--\ref{fig3} we show some graphics of different Julia sets obtained when Schr\"oder's method is applied to polynomials $(z-a)^m(z-b)^n$, $m\ge n\ge 1$. We compare these dynamical planes with the ones obtained for Newton's.   For instance, in Figure~\ref{fig1} we show the behavior when $p=m/n$ is increasing. We appreciate how the Julia set for Schr\"oder's method (a circle) tends to to collapse in the point $z=-1$  that in this case is the simple root. In the case of Newton's method, the Julia set is a kind of ``deformed parabola'', whose ``axis of symmetry''  is the real line, it is open to the left, the ``vertex'' tends to the simple root $z=-1$ and the ``latus rectum''  tends to zero. We see how the basin of attraction of the multiple root $z=1$ invades more and more the basin of the simple root $z=-1$, as it was pointed out by Hern\'andez-Paricio et al. \cite{Gut}.

In Figure~\ref{fig2} we see what happens when $p=m/n\approx 1$. The Julia set for Schr\"oder's method are circles getting bigger as $p$ approaches the value 1 and exploding into a half-plane limited by the imaginary axis when $p=1$. In the case of Newton's method, the Julia set is again a ``deformed parabola'' with  the real line as   ``axis of symmetry''  and open to the left. However as $p$ goes to 1, the ``vertex'' tends to  $z=0$ and the ``latus rectum''  tends to infinity. As a limit case, when $p=1$ this ``deformed parabola'' becomes a straight line, actually the imaginary axis.

Figure~\ref{fig3} shows the circle corresponding to the Julia of Schr\"oder's method applied to polynomials $(z-1)^m(z+1)^n$ with $p=m/n=2$. We can also see the Julia set of Newton's method applied to such polynomials. In the case of Newton's method we observe that the behavior is not the same fo values of $m$ and $n$ such that $p=m/n=2$. The corresponding ``deformed parabola'' tends to be smoother when the values of $m$ and $n$ increase.

Finally in Figure~\ref{fig4} we show the Julia set $J_{m,n,a,b}$ defined in Theorem~\ref{teo2} in the case $m=2$, $n=1$, $a=1$, $b=i$ together with the corresponding Julia set for Newton's method. In these figures we appreciate the loss of symmetry respect to the imaginary axis. This role is now played by the equidistant line between the roots $a$ and $b$.

As a further work, we would like to explore the  influence of the multiplicity in the Julia set of Newton's method applied to polynomials $(z-a)^m(z-b)^n$, $m\ge n\ge 1$ and its possible relationship between the study of Schr\"oder's method. In particular, we are interested in characterize the main properties of the ``deformed parabolas'' that appear in the case of Newton's method.

\begin{figure}[H]
\centering 
\twofractals{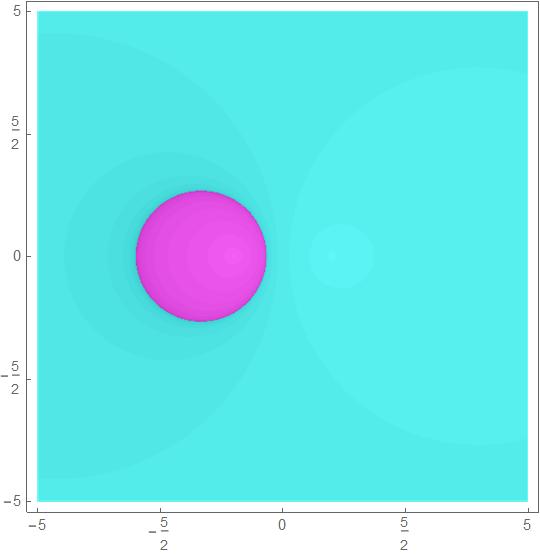}{Schr\"oder $m=2, \, n=1.$}{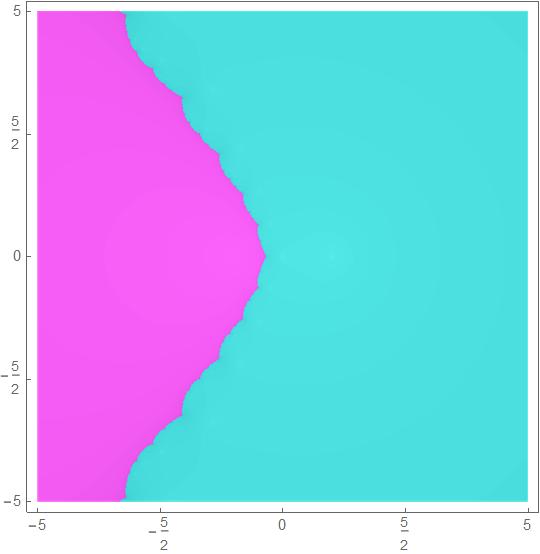}{Newton $m=2, \, n=1.$}

\twofractals{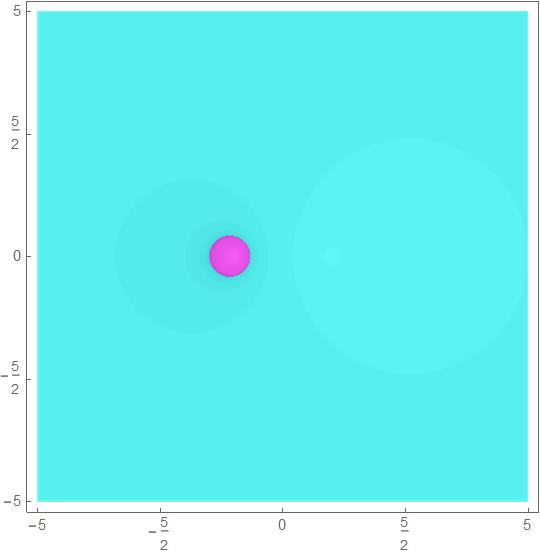}{Schr\"oder $m=5, \, n=1.$}{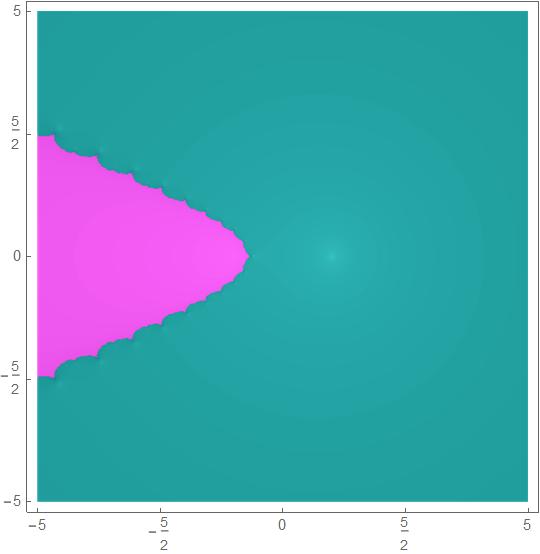}{Newton $m=5, \, n=1.$}

\twofractals{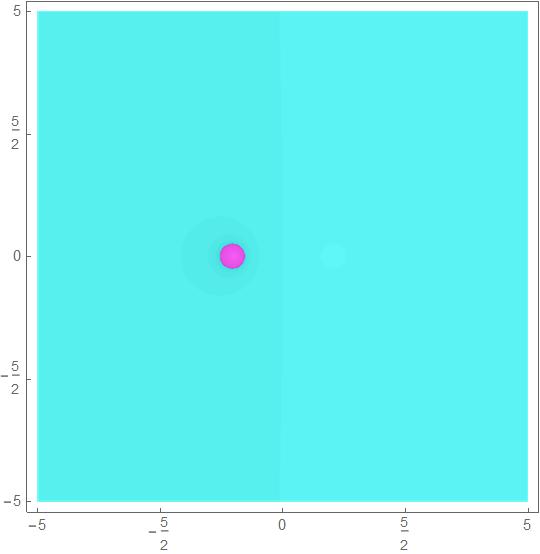}{Schr\"oder $m=8, \, n=1.$}{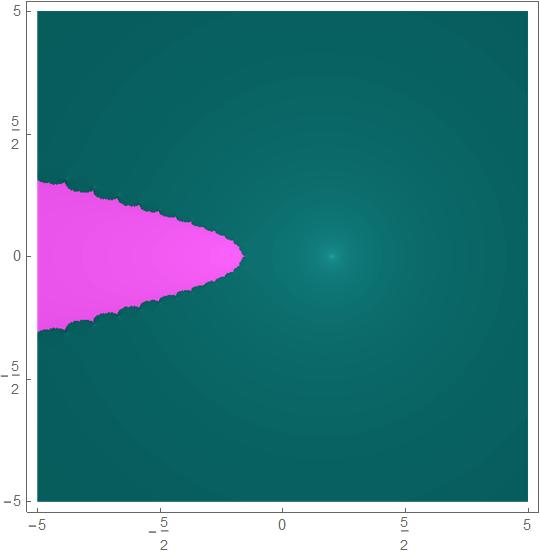}{Newton $m=8, \, n=1.$}
\caption{Basins of attraction of Schr\"oder's  and Newton's methods applied to polynomials $(z-1)^n(z+1)^m$ for $n=1$, $m=2, 5, 8$.}
\label{fig1}
\end{figure}  

\begin{figure}[H]
\centering 
\twofractals{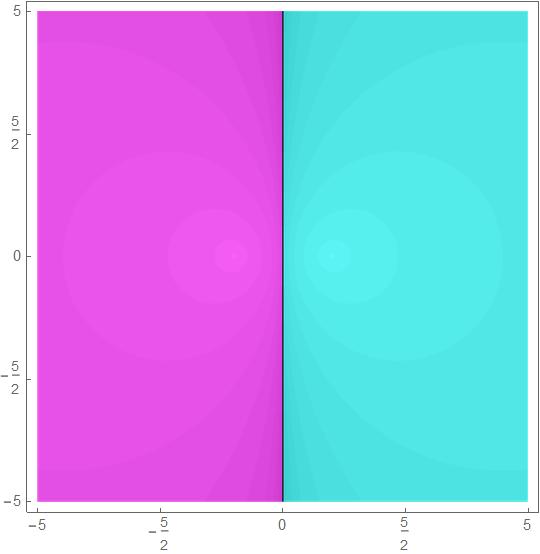}{Schr\"oder $m=6, \, n=6.$}{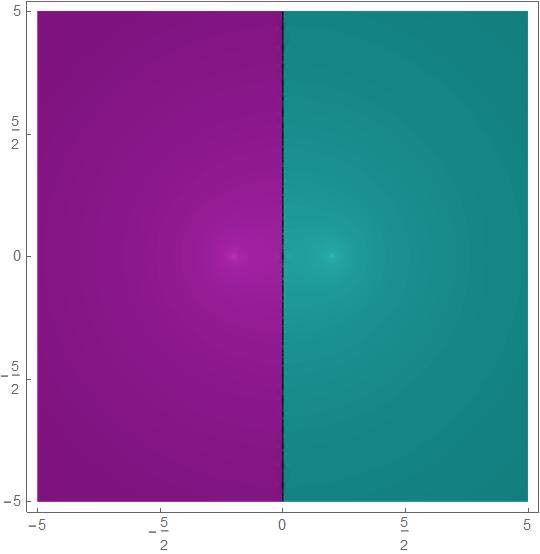}{Newton $m=6, \, n=6.$}

\twofractals{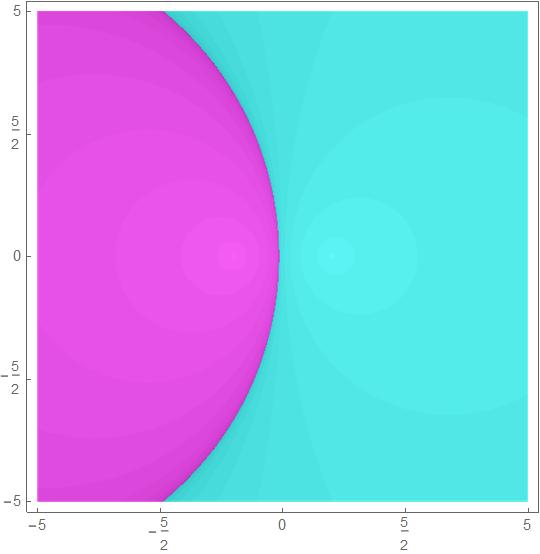}{Schr\"oder $m=7, \, n=6.$}{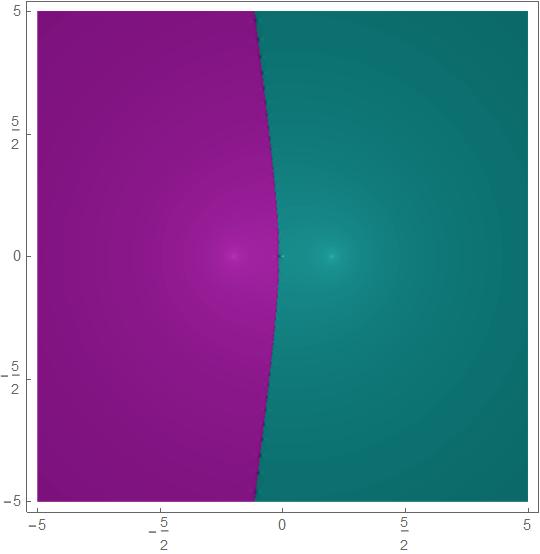}{Newton $m=7, \, n=6.$}

\twofractals{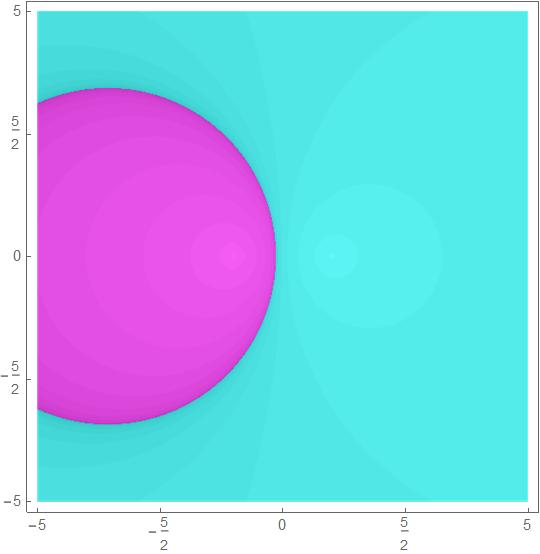}{Schr\"oder $m=8, \, n=6.$}{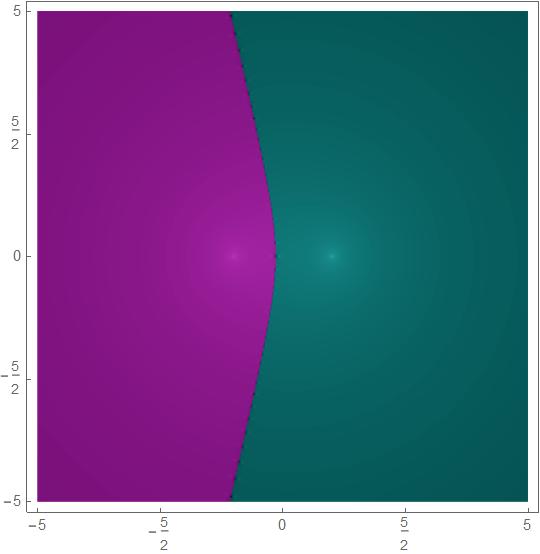}{Newton $m=8, \, n=6.$}
\caption{Basins of attraction of Schr\"oder's  and Newton's methods applied to polynomials $(z-1)^n(z+1)^m$ for $n=6$, $m=6, 7, 8$.}
\label{fig2}
\end{figure}

\begin{figure}[H]
\centering 
\imagen{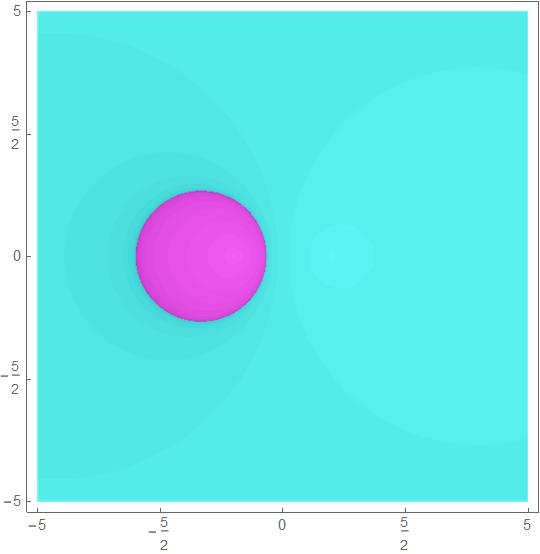}{Schr\"oder $p=m/n=2.$}{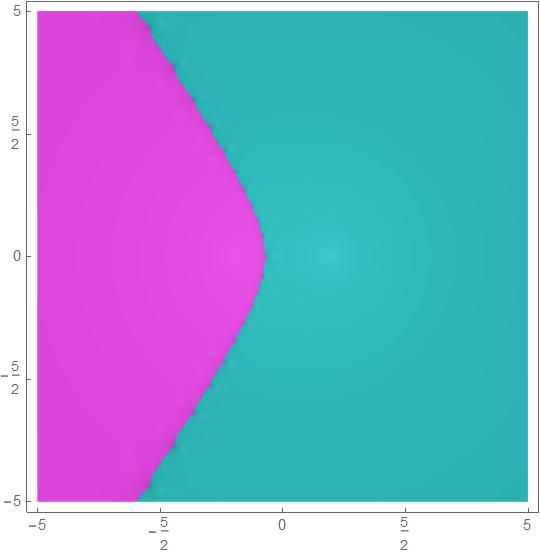}{Newton $m=4, \, n=2.$}

\imagen{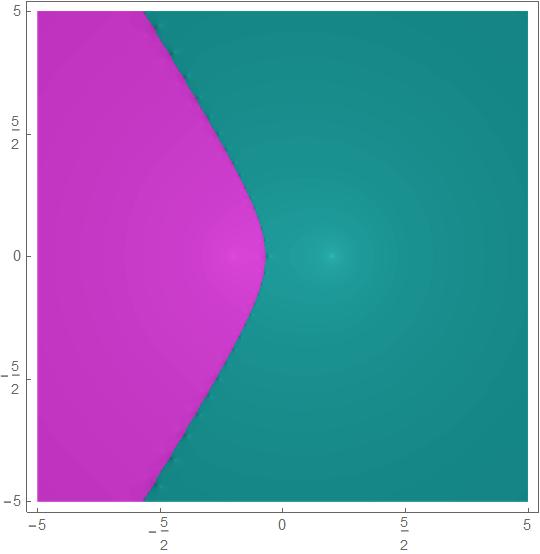}{Newton $m=6, \, n=3.$}{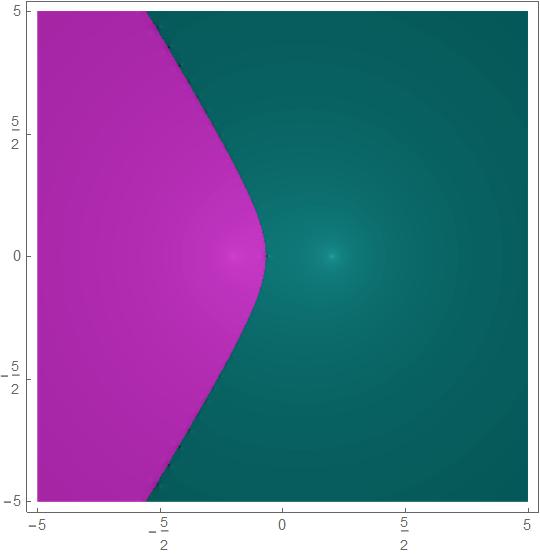}{Newton $m=8, \, n=4.$}

\caption{The first graphic shows the basin of attraction of Schr\"oder's method applied to polynomials $(z-1)^n(z+1)^m$ with $p=m/n=2$. The other graphics show the basins of attraction of  Newton's method applied to the same polynomials  for different values of $m$ and $n$ with $p=m/n=2$.}
\label{fig3}
\end{figure}  

\begin{figure}[H]
\centering 
\includegraphics[width=0.4\textwidth]{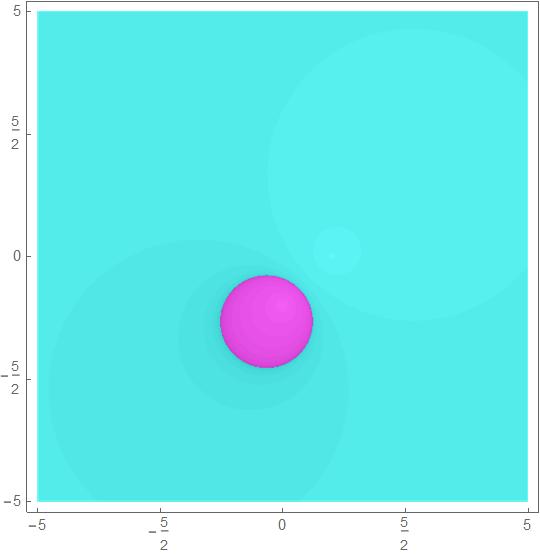}\quad 
\includegraphics[width=0.4\textwidth]{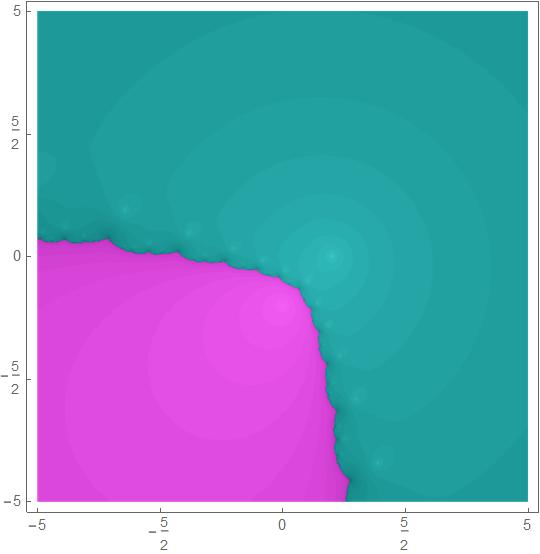}


\caption{Basin of attraction of Schr\"oder's and Newton's methods applied to polynomials $(z-1)^2(z+i)$.}
\label{fig4}
\end{figure}  
\newpage



\begin{thebibliography}{plain}


\bibitem[Amat (2003)]{Amat}
Amat, S., Busquier, S. and Guti\'errez, J. M. Geometric constructions of iterative functions to solve nonlinear equations. {\em J. Comput. Appl. Math.} {\bf 2003}, {\em 157}, 197--205.

\bibitem[Cayley (1897)]{Cay}
Cayley, A. The Newton-Fourier imaginary problem. {\em Amer. J.  Math.} {\bf 1897}, {\em 2}, 97.

\bibitem[Devaney (20203)]{Dev}
Devaney, R. L. {\em An Introduction to Chaotic Dynamical Systems, Second Edition}; Westview Press, Cambridge, 2003.

\bibitem[Dubeau (2013)]{Dub}
Dubeau, F. and Gnang,  C. On the Chebyshev-Halley family of iteration functions and the $n$-th root computation problem. {\em Intern. J. Comput. Pure Appl. Math.} {\bf 2013}, {\em 85}, 1051--1059.

\bibitem[Gutierrez (2014)]{Gut}
Guti\'errez, J. M., Hern\'andez-Paricio, L. J., Mara\~n\'on-Grandes, M., Rivas-Rodr{\'\i}guez, M.T.  Influence of the multiplicity of the roots on the basins of attraction of Newton's method. {\em Numer. Algor.} {\bf 2014}, {\em 66}, 431--455.



\bibitem[Petkovic (2010)]{Pet1}
Petkovi\'c, M. S., Petkovi\'c, L. D. and Herceg, \DJ. On Schr\"oder's families of root-finding methods. {\em J. Comput. Appl. Math.} {\bf 2010}, {\em 233}, 1755--1762.

\bibitem[Petkovic (2017)]{Pet2}
Petkovi\'c, M. S. and  Petkovi\'c, L. D. Dynamic study of Schr\"oder's families of first and second kind. {\em Numer. Algor.} {\bf 2018}, {\em 78}, 847--865.

\bibitem[Proinov (2013)]{Proinov}
Proinov, P. D. and Ivanov, S. I. Convergence of Schr\"oder's method for polynomial zeros of unknown multiplicity. {\em Comptes rendus de l'Acad\'emie bulgare des sciences: sciences math\'ematiques et naturelles} {\bf 2013}, {\em 66}, 1073--1080.

\bibitem[Scavo (1995)]{Scavo}
Scavo T. R. and Thoo, J. B. On the geometry of Halley's method. {\em Amer. Math. Monthly} {\bf 1995}, {\em 102}, 417--426.

\bibitem[Schr\"oder (1870)]{Sch}
Schr\"oder, E. \"Uber unendliche viele Algorithmen zur Aufl\"osung der Gleichungen. {\em Mathematische Annalen} {\bf 1870}, {\em 2}, 317--365.

\bibitem[Stewart (1993)]{Ste}
Stewart, G. W.  On Infinitely Many Algorithms for Solving Equations (English translation of Schr\"oder's original paper). {\em College Park, MD: University of Maryland, Institute for Advanced Computer Studies, Department of Computer Science} {\bf 1993}.

\bibitem[Werner (1981)]{Wer}
Werner, W.  Some improvements of classical iterative methods for the solution of nonlinear equations. {\em Numerical Solution of Nonlinear equations, Lecture Notes in Math, 878,} {\bf 1981}, {\em 878}, 427--440.
%
\end{thebibliography}
\end{document}